\documentclass[11pt]{amsart}
\usepackage{amsmath,amsthm,amssymb, amscd, amsfonts}
\usepackage[all]{xy}
\usepackage{leftidx}
\usepackage[latin1]{inputenc}
\usepackage{enumerate}

\theoremstyle{plain}

\newtheorem{theorem}{Theorem}[section]
\newtheorem{corollary}[theorem]{Corollary}

\newtheorem{proposition}[theorem]{Proposition}
\newtheorem{lemma}[theorem]{Lemma}

\theoremstyle{definition}
\newtheorem{definition}[theorem]{Definition}

\newtheorem{example}[theorem]{Example}

\theoremstyle{remark}
\newtheorem{remark}[theorem]{Remark}

\numberwithin{equation}{section}\theoremstyle{plain}

\newcommand{\I}{\mathcal{I}}

\renewcommand{\1}{\textbf{1}}

\newcommand{\B}{{\mathcal B}}
\newcommand{\C}{{\mathcal C}}
\newcommand{\D}{{\mathcal D}}

\newcommand{\Z}{{\mathcal Z}}
\newcommand{\Zz}{{\mathbb Z}}
\newcommand{\M}{\mathcal{M}}

\newcommand{\E}{{\mathcal E}}

\newcommand{\Rep}{\operatorname{Rep}}
\newcommand{\cd}{\mathrm{cd}}

\newcommand{\KER}{\mathfrak{Ker}}

\newcommand\Aut{\operatorname{Aut}}
\newcommand\Irr{\operatorname{Irr}}
\newcommand\FPdim{\operatorname{FPdim}}

\newcommand\vect{\operatorname{Vect}}
\newcommand\svect{\operatorname{sVect}}
\newcommand\id{\operatorname{id}}

\newcommand\Hom{\operatorname{Hom}}

\begin{document}
\title[The core a weakly group-theoretical braided fusion category]{The core a weakly group-theoretical braided fusion category}
\author{Sonia Natale}
\address{Facultad de Matem\'atica, Astronom\'\i a y F\'\i sica.
Universidad Nacional de C\'ordoba. CIEM -- CONICET. Ciudad
Universitaria. (5000) C\'ordoba, Argentina}
\email{natale@famaf.unc.edu.ar
\newline \indent \emph{URL:}\/ http://www.famaf.unc.edu.ar/$\sim$natale}
	
\thanks{This work was partially supported by  CONICET and SeCYT--UNC}
	
\keywords{Braided fusion category;  braided $G$-crossed fusion category; Tannakian category; core of a fusion category; weakly anisotropic braided fusion category; weakly group-theoretical fusion category}
	
\subjclass[2010]{18D10; 16T05}
	
\date{\today}
	
\begin{abstract} We show that the core of a weakly group-theoretical braided fusion category $\C$ is  equivalent as a braided fusion category to a tensor product  $\B \boxtimes \D$, where $\D$ is a pointed weakly anisotropic braided fusion category, and  $\B \cong \vect$ or $\B$ is an Ising braided category. In particular, if $\C$ is integral, then its core is a pointed weakly anisotropic braided fusion category.
As an application we give a characterization of the solvability of a weakly group-theoretical braided fusion category. We also prove that an integral modular category all of whose simple objects have Frobenius-Perron dimension at most 2 is necessarily group-theoretical.
\end{abstract}
	
\maketitle
	
\section{Introduction}

A braided fusion category is a fusion category endowed with a braiding, that is, a natural isomorphism $c_{X,Y} : X \otimes Y \rightarrow Y \otimes X$, $X, Y \in \C$, subject to the so-called hexagon axioms.  Braided fusion categories are of interest in many areas of mathematics and mathematical physics.

\medbreak
A generalization of the notion of a braided category is provided by that of a \emph{crossed braided category}, introduced by Turaev \cite{turaev-hqft}.
Let $G$ be a finite group. A \emph{braided $G$-crossed fusion
category} is a fusion category $\D$ endowed with a $G$-grading $\D
= \oplus_{g \in G}\D_g$ and an action of $G$ by tensor autoequivalences
$\rho:\underline G \to \Aut_{\otimes} \, \D$, such that $\rho^g(\D_h)
\subseteq \D_{ghg^{-1}}$, for all $g, h \in G$, and a $G$-braiding $c: X \otimes Y \to
\rho^g(Y) \otimes X$, $g \in G$, $X \in \D_g$, $Y \in \D$, subject to appropriate
compatibility conditions.

\medbreak
A braided fusion category is called \emph{Tannakian}, if it is equivalent as a braided fusion category to the category $\Rep G$ of finite dimensional representations of some finite group $G$, where the braiding is given by the usual flip of vector spaces.

The structure of braided fusion categories containing a Tannakian subcategory $\E \cong \Rep(G)$ can be described in terms of equivariantizations
of $G$-crossed braided fusion categories; see  \cite{mueger-crossed}, \cite{kirillov}, \cite[Section 4.4]{DGNOI}.

\medbreak
The \emph{core} of a braided fusion category was introduced in \cite{DGNOI}. As a braided fusion category, the core of a braided fusion category $\C$ is the neutral homogeneous component $\C^0_G$ of the de-equivariantization of $\C$ by a maximal Tannakian subcategory $\E \cong \Rep G$. The core of $\C$ is independent of $\E$. Furthermore, the core of a braided fusion category is \emph{weakly anisotropic}, that is, it contains no Tannakian subcategories stable under all braided auto-equivalences. In addition, the core of $\C$ is non-degenerate if $\C$ is non-degenerate.
The complete classification of pointed weakly anisotropic braided fusion categories has been proposed in \cite[Subsection 5.6.1]{DGNOI}.

\medbreak
Let $\C$ be a fusion category. Recall that the Frobenius-Perron dimension of a
simple object $X \in \C$ is defined as the Frobenius-Perron eigenvalue of
the matrix of left multiplication by the class of $X$ in the basis $\Irr(\C)$ of
the Grothendieck ring of $\C$ consisting of isomorphism classes of simple
objects. The Frobenius-Perron dimension of $\C$ is $\FPdim \C =
\sum_{X \in \Irr(\C)} (\FPdim X)^2$. The fusion category $\C$ is called \emph{integral}  if
$\FPdim X$ is a natural number, for all simple object $X \in \C$.

\medbreak A fusion category $\C$ is called \emph{weakly group-theoretical} if $\C$ is categorically Morita equivalent to a nilpotent fusion category (see Subsection \ref{s-nilp} for an overview of these notions). Every weakly group-theoretical fusion category has integer Frobenius-Perron dimension. It is conjectured that every fusion category of integer Frobenius-Perron dimension is weakly group-theoretical \cite{ENO2}.

\medbreak
Recall that a fusion category $\C$ is called \emph{pointed} if the Frobenius-Perron dimension of every simple object of $\C$ is $1$. A fusion category of Frobenius-Perron dimension $4$ which is not pointed is called an \emph{Ising category}. Any Ising category has two simple objects of Frobenius-Perron dimension $1$ and a third simple object of Frobenius-Perron dimension $\sqrt{2}$.  An \emph{Ising braided category} is an Ising fusion category endowed with a braiding. Ising braided categories and their spherical structures are classified in \cite[Appendix B]{DGNOI}; in particular, every Ising braided category is anisotropic and non-degenerate.

The main result of this paper is the following theorem, that describes the core of a weakly group-theoretical braided fusion category.

\begin{theorem}\label{core-wgt} Let $\C$ be a weakly group-theoretical braided fusion category. Then the core of $\C$ is equivalent as a braided fusion category  to a Deligne tensor product	 $\B \boxtimes \D$,
where $\D$ is a pointed weakly anisotropic braided fusion category and either $\B \cong \vect$ or $\B$ is an Ising braided category.

In particular, if $\C$ is integral, then its core is a pointed weakly anisotropic braided fusion category.
\end{theorem}

Theorem \ref{core-wgt} will be proved in Section \ref{core}. The theorem implies the fact, established in \cite{witt-wgt}, that the class of a weakly group-theoretical fusion category in the Witt group of non-degenerate braided fusion categories belongs to the subgroup generated by classes of non-degenerate pointed braided fusion categories and Ising braided categories.

As a consequence of Theorem \ref{core-wgt} we find that, for every weakly group-theoretical braided fusion category $\C$, there is a finite group $G$ such that $\C$ is equivalent to the $G$-equivariantization of a $G$-crossed braided fusion category whose neutral component is either pointed or the tensor product of a pointed braided fusion category and an Ising braided category.  In particular, the centralizer $\E'$ of any maximal Tannakian subcategory $\E$ of an integral weakly group-theoretical braided fusion category is group-theoretical, and the de-equivariantization by such a maximal Tannakian subcategory is a 2-step nilpotent crossed braided fusion category.

\medbreak
Theorem \ref{core-wgt} also alow us to give a characterization of the solvability of a weakly group-theoretical braided fusion category in terms of the solvability of its Tannakian subcategories (Theorem \ref{cor-sol}). This characterization is applied to show, on the one hand, that certain class of non-degenerate braided fusion categories  are solvable (Proposition \ref{paqbd}). In particular, we get that non-degenerate braided fusion categories of Frobenius-Perron dimension $p^aq^bc$, where $c$ is a square-free integer, are solvable (Corollary \ref{paqbd-solv}). On the other hand, we use the mentioned characterization to show that the solvability of a weakly group-theoretical fusion category is determined by the $S$-matrix of its Drinfeld center (Theorem \ref{fr-wgt}).

\medbreak
Also as a consequence of Theorem \ref{core-wgt}, we obtain the following result:

\begin{theorem}\label{cd12} Let $\C$ be an integral non-degenerate braided fusion category such that $\FPdim X \leq 2$, for every simple object $X$ of $\C$. Then $\C$ is group-theoretical.
\end{theorem}
	
Theorem \ref{cd12} will be proved in Section \ref{s-cd12}. Recall that a fusion category $\C$ is group-theoretical if it is categorically Morita equivalent to a pointed fusion category. Group-theoretical fusion categories are completely classified in terms of finite groups and their cohomology.

\medbreak
The paper is organized as follows. In Section \ref{s-pre} we recall the relevant notions on fusion categories used throughout the paper. In Section \ref{prev-core} we prove the existence of nontrivial Tannakian subcategories in certain integral  weakly group-theoretical braided fusion categories. This result is applied in the proof of Theorem \ref{core-wgt}, that we give in Section \ref{core}; some consequences of Theorem \ref{core-wgt} are also given in this section. In Section \ref{s-solv-wgt} we discuss some conditions that guarantee the solvability of a weakly group-theoretical braided fusion category. Finally, in Section \ref{s-cd12} we study non-degenerate integral braided fusion categories with Frobenius-Perron dimensions of simple objects at most $2$, and give a proof of Theorem \ref{cd12}.
	
\section{Preliminaries}\label{s-pre}
	
We shall work over an algebraically closed field $k$ of characteristic zero. The
category of finite dimensional vector spaces
over $k$ will be denoted by $\vect$.  A fusion category over $k$ is a semisimple tensor category over $k$ with
finitely many isomorphism classes of simple objects.
We refer the reader to \cite{ENO}, \cite{ENO2}, \cite{DGNOI} for the notions on fusion categories and braided fusion categories used throughout.

\medbreak
Let $\C$ be a fusion category. We shall denote by $\Irr(\C)$ the set of isomorphism classes of simple objects of $\C$. We shall use the notation $\cd(\C)$ to indicate the set of Frobenius-Perron dimensions of simple objects of $\C$, that is,
$$\cd(\C) = \{ \FPdim X: \, X \in \Irr(\C)\}.$$

The fusion subcategory of $\C$ generated by objects $X_1, \dots, X_n$ of $\C$ will be denoted $\langle X_1, \dots, X_n\rangle$; this is the smallest fusion subcategory containing the objects $X_1, \dots, X_n$. The fusion subcategory generated by fusion subcategories $\D_1, \dots, \D_n$ will also be denoted $\D_1 \vee \dots \vee \D_n$.

The group of isomorphism classes of invertible objects of $\C$ will be indicated by $G(\C)$. The largest pointed subcategory of $\C$, denoted $\C_{pt}$, is thus the fusion subcategory generated by $G(\C)$.
	
\medbreak
Let $\C, \D$ be fusion categories and let $F:\C \to \D$ be a tensor functor.
The functor $F$ is called \emph{dominant} if every object $Y$ of $\D$ is a subobject of $F(X)$ for some object $X$ of $\C$.

We shall denote by $\KER_F$ the fusion subcategory of $\C$ whose objects are those $X \in \C$ such that $F(X)$ is a trivial object of $\D$, that is, a direct sum of copies of the unit object of $\D$. The functor $F$ is \emph{normal} if every simple object $X$ of $\C$ such that $\Hom_\D(\1, X) \neq 0$ belongs to $\KER_F$.
	
\medbreak The fusion subcategory  of $\D$ generated by the essential image of $F$ will be denoted $F(\C)$.  Thus $F(\C)$ is generated as an additive category by the objects $Y$ of $\D$ which are subobjects of $F(X)$, for some object $X$ of $\C$. Observe that, if $F: \C \to \D$ is any tensor functor between fusion categories $\C, \D$, then the corestriction of $F$  is a dominant tensor functor   $\C \to F(\C)$.

\medbreak A $\C$-module category is a finite semisimple $k$-linear abelian category $\M$ endowed with a bifunctor $\M \times \C \to \M$ satisfying, up to coherent natural isomorphisms, the usual associativity and unit axioms for an action. If $\M$ is an \emph{indecomposable} $\C$-module category, i.e. $\M$ cannot be decomposed into a direct sum of proper $\C$-module subcategories, then the category $\textrm{Fun}_\C(\M, \M)$ of $\C$-module endofunctors of $\M$ is a fusion category.

Two fusion categories $\C$ and $\D$ are called \emph{categorically Morita equivalent} if there exists an equivalence of fusion categories $\D^{op} \cong \textrm{Fun}_\C(\M, \M)$, for some indecomposable $\C$-module category $\M$. By \cite[Theorem 3.1]{ENO2}, $\C$ and $\D$ are called categorically Morita equivalent if and only if its Drinfeld centers $\Z(\C)$ and $\Z(\D)$ are equivalent as braided fusion categories.
	
\subsection{Nilpotent, weakly group-theoretical and solvable fusion categories}\label{s-nilp} Let $G$ be
a finite group and let $\C$ be a fusion category. A $G$-grading on a fusion category $\C$ is a decomposition $\C =
\oplus_{g\in G} \C_g$, such that $\C_g \otimes \C_h \subseteq \C_{gh}$, for all $g, h \in G$.
The fusion category $\C$ is called a \emph{$G$-extension} of a
fusion category $\D$ if there is a faithful grading $\C = \oplus_{g\in G} \C_g$
with neutral component $\C_e \cong \D$.
	
\medbreak Let $\C_{ad}$ be the \emph{adjoint subcategory} of $\C$, that is, the fusion subcategory generated by $X\otimes X^*$, where $X$ runs over the simple objects of $\C$. Then the fusion category $\C$ has a canonical faithful grading $\C =
\oplus_{g \in U(\C)}\C_g$, with neutral component $\C_e = \C_{ad}$. The group $U(\C)$ is called the
\emph{universal grading group} of $\C$. See \cite{gel-nik}.

\medbreak
Let $\C$ be a fusion category such that $\FPdim \C \in \Zz$. By \cite[Theorem 3.10]{gel-nik}, there exists an elementary abelian 2-group $E$ and a set of pairwise distinct square-free natural numbers $n_x$, $x \in E$, with $n_0 = 1$, such that
$\C$ is endowed with a faithful $E$-grading $$\C=\oplus_{x\in E}\C_x,$$ where, for all $x\in E$, $\C_x$ is the full subcategory whose simple objects $X$ satisfy $\FPdim X \in \Zz \sqrt{n_x}$.

In particular, the neutral component $\C_0 = \C_{int}$ of this grading is the unique maximal integral fusion subcategory of $\C$.

\medbreak The \emph{descending central series} of $\C$ is the series of fusion subcategories
\begin{equation}
\dots \subseteq \C^{(n+1)} \subseteq \C^{(n)} \subseteq \dots \subseteq \C^{(1)}
\subseteq \C^{(0)} = \C,
\end{equation}
defined recursively as $\C^{(n+1)} = (\C^{(n)})_{ad}$, for all $n \geq 0$.

\medbreak Observe that if $\C = \oplus_{g\in G}\C_g$ is any grading on a fusion category $\C$, then $\C_{ad} \subseteq \C_e$. In particular, if $\FPdim \C \in \Zz$, then $\C_{ad} \subseteq \C_{int}$, that is, $\C_{ad}$ is an integral fusion subcategory.

\begin{remark} The fusion subcategories $\C^{(n)}$, $n \geq 0$, are stable under any tensor autoequivalence of $\C$. In addition, if $\FPdim \C \in \Zz$, then $\C_{int}$ is also stable under any tensor autoequivalence of $\C$.
\end{remark}

\begin{remark}\label{gding-sub} Suppose that $\C$ is a fusion category that admits a faithful grading $\C = \oplus_{g\in G} \C_g$. If $\D \subseteq \C$ is a fusion subcategory, then $\D$ has a faithful grading $\D = \oplus_{h\in H} \D_h$, where $H$ is the subgroup of $G$ defined by $H = \{g \in G:\; \D\cap \C_g \neq 0\}$, and $\D_h = \D \cap \C_h$, for all $h \in H$.
	
In particular, if $\D$ is such that $\D_{ad} = \D$, then necessarily $H = \{ e\}$, hence $\D \subseteq \C_e$.	
\end{remark}

\medbreak The fusion category $\C$ is \emph{nilpotent} if there exists $n \geq 0$ such that $\C^{(n)} \cong \vect$. If $\C$ is categorically Morita equivalent to a nilpotent fusion category, then $\C$ is called \emph{weakly group-theoretical}.

It is known that the class of weakly group-theoretical fusion categories is stable under taking fusion subcategories, component categories in quotient categories, Morita equivalent categories, tensor products, Drinfeld centers, equivariantizations and group extensions \cite[Proposition 4.1]{ENO2}.

\medbreak
A fusion category $\C$ is called \emph{cyclically nilpotent} if there exists a series of fusion subcategories $\vect\cong \C_0 \subseteq \dots \subseteq \C_N = \C$, such that for all $i = 0, \dots, N-1$, $\C_{i+1}$ is a $\mathbb Z_{p_i}$-extension of $\C_{i}$, for some prime numbers $p_1, \dots, p_N$.
If a weakly group-theoretical fusion category is categorically Morita equivalent to a cyclically nilpotent fusion category, then it is called \emph{solvable}. By \cite[Proposition 4.5]{ENO2}, the class of solvable categories is closed under taking extensions and equivariantizations by solvable groups, Morita equivalent categories, tensor products, Drinfeld center, fusion subcategories and component categories of quotient categories.
In addition, for every finite group $G$, the fusion category $\Rep G$ is solvable if and only if the group $G$ is solvable.

\subsection{Group actions and equivariantizations}\label{ss-equiv}

Consider  an action of a finite group $G$ on a fusion
category $\C$ by tensor autoequivalences $\rho: \underline
G \to \underline \Aut_{\otimes} \, \C$.
The \emph{equivariantization} of $\C$ with respect to the action $\rho$, denoted
$\C^G$, is a fusion category whose objects are pairs  $(X, \mu)$, such that $X$
is an object of $\C$ and $\mu = (\mu^g)_{g \in G}$, is a collection of
isomorphisms $\mu^g:\rho^gX \to X$, $g \in G$, satisfying appropriate
compatibility conditions.

\medbreak
The forgetful functor $F: \C^G \to \C$, $F(X, \mu) = X$,
is a normal dominant tensor functor that gives rise to an exact sequence of
fusion categories $$\Rep G \longrightarrow \C^G \overset{F}\longrightarrow \C.$$ See \cite{tensor-exact}.

\medbreak Simple objects of $\C^G$ are parameterized by pairs $(Y, \pi)$, where $Y$ runs over the $G$-orbits on $\Irr(\C)$ and $\pi$ is an equivalence class of an  irreducible $\alpha_Y$-projective representation of the inertia subgroup
$$G_Y = \{g\in G:\,  \rho^g(Y) \cong  Y \},$$ for certain 2-cocycle $\alpha_Y: G_Y \times G_Y \to k^{\times}$ \cite[Corollary 2.13]{fr-equiv}. We shall use the notation $S_{Y, \pi}$ to indicate the isomorphism class of the simple object corresponding to the pair $(Y, \pi)$. Then the dimension of $S_{Y, \pi}$ is given by the formula
\begin{equation}\label{dim-equiv}\FPdim S_{Y, \pi} =  [G:G_Y] \dim \pi \FPdim Y,\end{equation}
and we have an isomorphism
\begin{equation}F(S_{Y, \pi}) \cong \bigoplus_{g\in G/G_Y} \, \rho^g(Y)^{(\dim \pi)}.
\end{equation}

As a consequence of \eqref{dim-equiv}, we have the following lemma that will be needed in the course of the proof of Theorem \ref{cd12} in Section \ref{s-cd12} (see Lemma \ref{repu-nil}).

\begin{lemma}\label{equiv-nil} Let $\B$ be an integral fusion category such that $\FPdim \B = 2^n$, $n \geq 1$. Let also $G$ be a finite abelian group and let $\rho: \underline G \to \Aut_{\otimes}\B$ be an action of $G$ on $\B$ by tensor autoequivalences such that the Grothendieck ring $K_0(\B^G)$ is commutative and  $\cd(\B^G) \subseteq \{1, 2\}$. Then $\B^G$ is nilpotent. \end{lemma}

\begin{proof} Suppose that $X_1, \dots, X_n$, $n \geq 1$, are simple objects of $\B^G$ that generate a fusion subcategory $\D$. The commutativity of the Grothendieck ring $K_0(\B^G)$ guarantees that the simple constituents of the objects $X_1\otimes X_1^*, \dots, X_n\otimes X_n^*$, generate the adjoint subcategory $\D_{ad}$.

\medbreak
We may assume that $\B^G$ is not pointed.
Let $F: \B^G \to \B$ denote the canonical normal tensor functor such that $\KER_F \cong  \Rep G$.

\medbreak We first claim that $F((\B^G)^{(n)}) \subseteq \B^{(n)} \vee \B_{pt}$, for all $n \geq 0$. To prove this claim, we argue by induction on $n$ as follows. If $n = 0$ there is nothing to prove. Let $n \geq 0$ and let $X$ be a simple object of $(\B^G)^{(n+1)}$.  Since $(\B^G)^{(n+1)} = ((\B^G)^{(n)})_{ad}$, there exist simple objects  $X_1, \dots, X_m$ of $(\B^G)^{(n)}$ such that $X$ is a simple constituent of $X_1\otimes X_1^* \otimes \dots \otimes X_m\otimes X_m^*$.

By induction, $F(X_1), \dots, F(X_m) \in \B^{(n)} \vee \B_{pt}$. Moreover, for every $1\leq i \leq m$, we have that either $F(X_i)$ is simple or else $F(X_i) \in \B_{pt}$. Hence
$$F(X_1\otimes X_1^* \otimes \dots \otimes X_m\otimes X_m^*) \cong F(X_1)\otimes F(X_1)^* \otimes \dots \otimes F(X_m)\otimes F(X_m)^*$$
belongs to $\B^{(n+1)} \vee \B_{pt}$.
Since $F(X)$ is a direct summand of $F(X_1\otimes X_1^*) \otimes \dots \otimes F(X_m\otimes X_m^*)$, then $F(X) \in \B^{(n+1)} \vee \B_{pt}$. Thus $F((\B^G)^{(n+1)}) \subseteq \B^{(n+1)} \vee \B_{pt}$, and the claim follows.

\medbreak Since $\FPdim \B = 2^n$, then $\B$ is nilpotent \cite[Theorem 8.28]{ENO}. Therefore there exists $N \geq 1$ such that $\B^{(N)} \cong \vect$. Let $\D = (\B^G)^{(N)}$ so that, in view of the previous claim, $F(\D) \subseteq \B_{pt}$. We shall show next that the category $\D$ is nilpotent, which will imply that $\B^G$ is nilpotent.

\medbreak
Let $X$ be a simple object of $\D$.
Suppose that $X = S_{Y, \pi}$ corresponds to a pair $(Y, \pi)$, where $Y$ is a simple object of $\B$ and $\pi$ is an irreducible $\alpha_Y$-representation of the inertia subgroup $G_Y$.
Assume that $X$ is not invertible, that is, $\FPdim X = 2$. Since $F(\D) \subseteq \B_{pt}$, then we have $\FPdim Y = 1$.
In view of \eqref{dim-equiv}, this implies that $F(X) \cong \bigoplus_{g \in G/G_Y}\rho^g(Y)^{(e)}$, where $G = G_Y$ and $e = \dim \pi = 2$ or $G \neq G_Y$ and $e = \dim \pi = 1$.


Since $Y$ is invertible, then $Y \otimes Y^* \cong YY^{-1} \cong \1$. Thus, if $G = G_Y$, then $F(X \otimes X^*) \cong F(X) \otimes F(X^*) \cong  (Y \otimes Y^*)^{(4)} \cong \1^{(4)}$. This implies that $X \otimes X^* \in \KER_F \cong \Rep G$. Since $G$ is abelian, then we get that in this case $X \otimes X^* \in (\B^G)_{pt}$.

\medbreak
Suppose that $G \neq G_Y$, that is, $F(X) \cong Y \oplus \rho^g(Y)$, where $g \in G$ is such that  $\rho^g(Y) \ncong Y$. Then $F(X)^* \cong Y^{-1} \oplus \rho^g(Y)^{-1}$ and
\begin{equation}\label{adj-1}F(X) \otimes F(X)^* \cong \1^{(2)} \oplus \rho^g(Y) Y^{-1} \oplus Y \rho^g(Y)^{-1} = \1^{(2)} \oplus Z \oplus Z^{-1},\end{equation}
where $Z = \rho^g(Y) Y^{-1}$ is an invertible object of $\B$.

\medbreak
Decomposing $X \otimes X^*$ into a direct sum of simple objects, we find that either $X \otimes X^* \in (\B^G)_{pt}$ or $X \otimes X^* \cong \1 \oplus a \oplus X_1$, where $a \in G(\B^G)$ and $X_1$ is a 2-dimensional simple object. The decomposition \eqref{adj-1} implies that $F(X_1) \cong Z \oplus Z^{-1}$; indeed, since $F$ is normal, then the unit object $\1$ has multiplicity zero in $F(X_1)$ (otherwise $X_1 \in \KER_F$, which is impossible since $\KER_F$ is pointed).

\medbreak
We have thus shown that $\D_{ad}$ is contained in the fusion subcategory
$$(\B^G)_{pt} \vee \langle X \in \Irr(\B^G): \, F(X) \cong Z \oplus Z^{-1}, \text{ for some } Z \in G(\B) \rangle.$$
	
Observe that, if $X$ is a simple object such that $F(X) \cong  Z \oplus Z^{-1}$, for some $Z \in G(\B)$, then a simple constituent $X_1$ of $X \otimes X^*$ is either invertible or it satisfies $F(X_1) \cong  Z^2 \oplus Z^{-2}$: this follows from the relation $F(X \otimes X^*) \cong \1^{(2)} \oplus Z^2 \oplus Z^{-2}$ and the normality of $F$, since if $X_1$ is not invertible, the unit object $\1$ has multiplicity zero in $F(X_1)$.

It follows from an inductive argument that, for all $n \geq 1$, $\D^{(n)}$ is contained in the fusion subcategory
$$(\B^G)_{pt} \vee \langle X \in \Irr(\B^G): \, F(X) \cong Z^{2^{n-1}} \oplus Z^{-2^{n-1}}, \text{ for some } Z \in G(\B) \rangle.
$$

Since the group $G(\B)$ is a $2$-group, this implies that $\D^{(m)} \subseteq (\B^G)_{pt} \vee \KER_F$, for some $m$, and since $\KER_F \cong \Rep G$ is pointed, then $\D^{(m+1)} \cong \vect$ and therefore $\D$ is nilpotent, as claimed. As observed before, this implies that $\B^G$ is nilpotent and finishes the proof of the lemma.
\end{proof}

\subsection{Centralizers in braided fusion categories} Let $\C$ be a braided fusion category and let $\D$ be a fusion subcategory of  $\C$. The \emph{M\" uger centralizer} $\D'$ of $\D$ in $\C$ is the fusion subcategory of $\C$ generated by objects $X \in \C$ such that $c_{Y, X}c_{X, Y} = \id_{X \otimes Y}$, for all objects $Y \in \D$. The centralizer $\C'$ of $\C$ is called the \emph{M\" uger (or symmetric) center} of $\C$.

A braided fusion category $\C$ is called \emph{symmetric} if $\C' = \C$. A symmetric fusion category is called \emph{Tannakian} if $\C \cong \Rep G$ as braided fusion categories, for some finite group $G$, where the braiding in $\Rep G$ is given by the flip isomorphism $X \otimes Y \to Y \otimes X$.

\begin{remark} Let $\C, \tilde \C$ be braided fusion categories and let $F: \C \to \tilde\C$ be a braided tensor functor, that is,  $F$ is a tensor functor such that $F(c_{X, Y}) = F^2_{Y, X} \tilde c_{F(X), F(Y)} {F^2_{X, Y}}^{-1}$, for all $X, Y \in \C$, where $c$ and $\tilde c$ denote the braidings in $\C$ and $\tilde \C$, respectively, and  $F^2_{X, Y}: F(X) \otimes F(Y) \to F(X \otimes Y)$ is the monoidal structure on $F$.
Then, for every fusion subcategory $\D$ of $\C$, we have $F(\D') \subseteq F(\D)'$.  Hence, if $F$ is a braided equivalence, then $F(\D') = F(\D)'$. In particular, $F(\C') = \C'$, for every braided autoequivalence of $\C$.
\end{remark}
	
\medbreak  Given a symmetric fusion category $\C$, there exist a finite
group $G$ and a central element $u \in G$ of order $2$, such that $\C$ is equivalent to the category $\Rep(G, u)$ of representations of $G$ on finite-dimensional super-vector spaces where $u$ acts as the parity operator  \cite{deligne}.
If $\C = \Rep(G, u)$, then $\E = \Rep (G/(u))$ is the unique maximal Tannakian subcategory of $\C$ and there is a faithful $\Zz_2$-grading $\C = \E_0 \oplus \E_1$, with $\E_0 = \E$. In particular, if $\FPdim \C > 2$, then $\E_0$ is a nontrivial Tannakian subcategory of $\C$. In addition, if $\C$ is a symmetric fusion category of Frobenius-Perron dimension $2$ which is not Tannakian, then $\C$ is equivalent to the category $\svect$ of finite-dimensional super-vector spaces. See \cite[2.12]{DGNOI}.
	
\medbreak	
A braided fusion category $\C$ is  called \emph{non-degenerate}
if $\C' \cong \vect$, and it is called \emph{slightly degenerate} if $\C' \cong \svect$. For instance, the Drinfeld center $\Z(\D)$ of a fusion category $\D$ is a non-degenerate braided fusion category.

\medbreak If $\D\subseteq \C$ is a fusion subcategory such that $\D$ is non-degenerate, then $\D'$ is non-degenerate as well and $\C \cong \D \boxtimes \D'$ as braided fusion categories \cite[Theorem 4.2]{mueger-str}.
	
\subsection{Tannakian subcategories and braided crossed fusion categories}\label{2.4}
	
Let $G$ be a finite group. A \emph{braided $G$-crossed fusion
category} is a fusion category $\D$ endowed with a $G$-grading $\D
= \oplus_{g \in G}\D_g$ and an action of $G$ by tensor autoequivalences
$\rho:\underline G \to \Aut_{\otimes} \, \D$, such that $\rho^g(\D_h)
\subseteq \D_{ghg^{-1}}$, for all $g, h \in G$, and a $G$-braiding $c: X \otimes Y \to
\rho^g(Y) \otimes X$, $g \in G$, $X \in \D_g$, $Y \in \D$, subject to
compatibility conditions. See \cite{turaev-hqft}. If $\D$ is a braided $G$-crossed fusion category, then the equivariantization $\D^G$ is a braided fusion category, with braiding induced from the $G$-braiding in $\D$. Furthermore, the canonical embedding $\Rep G \to \D^G$ identifies $\Rep G$ with a Tannakian subcategory of $\D^G$.
	
\medbreak Conversely, suppose that $\C$ is a braided fusion category and $\E \cong \Rep G$ is a Tannakian subcategory of $\C$.
Then the de-equivariantization $\C_G$ of $\C$ with respect to $\E$ is a braided $G$-crossed fusion category in a canonical way, and there is an equivalence of braided fusion categories $\C \cong (\C_G)^G$.

In this way, equivariantization and de-equivariantization define inverse bijections between equivalence classes of braided fusion categories containing $\Rep G$ as a Tannakian subcategory and equivalence classes of $G$-crossed braided fusion categories \cite{mueger-crossed}, \cite{kirillov}, \cite[Section 4.4]{DGNOI}.

\medbreak
An instance of this correspondence occurs when $\C = \Z(\D)$ is the Drinfeld center of a fusion category $\D$, such that $\D$ is a $G$-extension of a fusion category $\D_0$. In this case $\Rep G$ is a Tannakian subcategory of $\C$ and the neutral homogeneous component of the $G$-crossed braided fusion category $\C_G$ is equivalent to the Drinfeld center $\Z(\D_0)$; see \cite{center-gdd}.

\medbreak 	
Let $\C$ be a braided fusion category and let $\E \cong \Rep G$ be a Tannakian subcategory of $\C$. Let also $\C_G^0$ denote the neutral component of $\C_G$ with respect to the associated $G$-grading. Then $\C_G^0$ is a braided fusion category and the crossed action of $G$ on $\C_G$ induces an action of $G$ on $\C_G^0$ by braided auto-equivalences. Moreover, there is an equivalence of braided fusion categories $(\C_G^0)^G \cong \E'$, where $\E$' is the centralizer in $\C$ of the Tannakian subcategory $\E$.
	
The braided fusion category $\C$ is non-degenerate if and only if $\C_G^0$ is non-degenerate and the $G$-grading of $\C_G$ is faithful \cite[Proposition 4.6 (ii)]{DGNOI}.  In this case there is an equivalence of braided fusion categories
$$\C \boxtimes \C_G^0 \cong \Z(\C_G).$$
See \cite[Corollary 3.30]{witt-ndeg}. In particular, if $\C$ is a non-degenerate braided fusion category containing a Tannakian subcategory $\E \cong \Rep G$, then
$$\dfrac{\FPdim \C}{|G|^2} = \FPdim \C_G^0.$$
Hence, if $\FPdim \C_G^0$ is an integer, then $\FPdim \C$ is an integer as well and $(\FPdim \E)^2 = |G|^2$ divides $\FPdim \C$.

\subsection{The core of a braided fusion category}\label{ss-core}

Let $\C$ be a braided fusion category. The \emph{core} of $\C$ was introduced in \cite{DGNOI}. As a braided fusion category, the core of $\C$ is the neutral homogeneous component $\C^0_G$ of the de-equivariantization of $\C$ by a maximal Tannakian subcategory $\E \cong \Rep G$. By \cite[Theorem 5.9]{DGNOI}, the core of $\C$ is independent of $\E$.

\begin{remark} Observe that a braided fusion category $\C$ is weakly group-theoretical if and only if its core is weakly group-theoretical.
\end{remark}

Recall from \cite{DGNOI} that $\C$ is called \emph{anisotropic} if $\C$ contains no nontrivial Tannakian subcategories and it is called \emph{weakly anisotropic} if it contains no nontrivial Tannakian subcategories stable under all braided auto-equivalences of $\C$.
It is shown in \cite[Corollaries 5.19 and 5.15]{DGNOI} that the core is a weakly anisotropic braided fusion category, and it is non-degenerate if $\C$ is non-degenerate.

\begin{example} The core of an anisotropic braided fusion category $\C$ is $\C$ itself. In particular, if $\I$ is an Ising braided category, then the core of $\I$ is $\I$.
On the other hand, if $\I_1$ and $\I_2$ are Ising braided categories, then the core of $\I_1 \boxtimes \I_2$ is a pointed braided fusion category; see \cite[Lemma B.24]{DGNOI}. 	
\end{example}

\begin{example} Suppose that $\D$ is a fusion category with \emph{fermionic Moore-Read fusion rules}, that is, the isomorphism classes of simple objects of $\C$ consist of four invertible objects $\1, g, g^2, g^3$, and two (dual) simple objects $X$ and $X'$ of Frobenius-Perron dimension $\sqrt{2}$,  such that $\D$ has commutative fusion rules determined by
$g^i\otimes g^j \cong g^{i+j (\textrm{mod }n)}$, $0\leq i, j \leq 3$, and
\begin{align*}& g^2 \otimes X \cong X, \quad g^2 \otimes X' \cong
X', \quad X \otimes X' \cong \1 \oplus g^2, \\
& g \otimes X \cong X', \quad g \otimes X' \cong
X, \quad X \otimes X \cong g \oplus g^3, \\
& g^3 \otimes X \cong X', \quad g^3 \otimes X' \cong X, \quad X'
\otimes X' \cong g \oplus g^3. \end{align*} See \cite{bonderson, liptrap}. Observe that the category $\D_{ad}$ is generated by the simple objects $\1$ and $g^2$, and it is equivalent to the category $\C(\Zz_2)$ of finite dimensional $\Zz_2$-graded vector spaces.

\medbreak
Let $\C = \Z(\D)$ be the Drinfeld center of $\D$, so that $\C$ is a non-degenerate braided fusion category of Frobenius-Perron dimension $(\FPdim \C)^2 = 64$, which is not integral.
It follows from \cite{center-gdd} that $\C$ contains a Tannakian subcategory $\E \cong \Rep U(\D)$ such that $\C_{U(\D)}^0 \cong \Z(\D_{ad}) \cong \Rep D(\Zz_2)$, where $D(\Zz_2)$ is the Drinfeld double of the group $\Zz_2$.
Since $U(\D)$ is of order $4$ and $\C$, not being integral, cannot contain Tannakian subcategories of dimension $8$, then $\E$ is a maximal Tannakian subcategory of $\C$.

Therefore the core of $\C$ coincides in this example with $\Rep D(\Zz_2)$. In particular, since $D(\Zz_2)$ is commutative, then the core of $\C$ is a pointed braided fusion category.
\end{example}
	
\begin{definition} Let $\C$ be a braided fusion category and let $\Gamma$ be a subgroup of the group $\Aut_{br}\C$ of braided autoequivalences of $\C$.
A fusion subcategory $\D$ of $\C$ will be called \emph{$\Gamma$-stable}  if $\sigma(\D) = \D$, for all $\sigma \in \Gamma$. If $\D$ is is $\Aut_{br}\C$-stable, then $\D$ will be called a  \emph{characteristic} fusion subcategory.

We shall say that $\C$ is \emph{$\Gamma$-anisotropic} if it contains no proper $\Gamma$-stable Tannakian subcategories.
\end{definition}

\begin{remark}\label{ppties} (i) It is clear that for every subgroup $\Gamma \subseteq \Aut_{br}\C$, every $\Gamma$-anisotropic braided fusion category is weakly anisotropic. Suppose that $\Gamma = \{ \id_\C\}$. Then a $\Gamma$-anisotropic braided fusion category is exactly an anisotropic braided fusion category in the terminology of \cite{DGNOI}.
	
\medbreak (ii) Let $\C$ be a weakly anisotropic braided fusion category. Suppose that $\D \subseteq \C$ is a characteristic fusion subcategory of $\C$. Then $\Gamma = \Aut_{br}\C$ is a subgroup of $\Aut_{br}\D$ and $\D$ is $\Gamma$-anisotropic; in particular,  $\D$ is also weakly anisotropic; see \cite[Lemma 5.26]{DGNOI}.

\medbreak
(iii) Let $\C$ be a braided fusion category and let $\Gamma$ be a subgroup of $\Aut_{br}\C$. If $\D$ is a $\Gamma$-stable fusion subcategory of $\C$, then the fusion subcategories $\D_{ad}$, $\D_{pt}$, $\D_{int}$, $\D'$ and $\D\cap \D'$ are also $\Gamma$-stable. In particular, the fusion subcategories $\C_{ad}$, $\C_{pt}$, $\C_{int}$, $\C'$, as well as their M\" uger centers, are characteristic subcategories of $\C$.
\end{remark}

\begin{lemma}{\cite[Lemma 5.27]{DGNOI}.}\label{wa-symm} Let $\C$ be a symmetric fusion category and suppose that $\C$ is weakly anisotropic. Then $\C \cong \vect$ or $\C \cong \svect$. \qed
\end{lemma}

\begin{proposition}\label{center-st} Let $\C$ be a weakly anisotropic braided fusion category and let $\D$ be a characteristic fusion subcategory of $\C$. Then either $\D$ is non-degenerate or $\D \cap \D' \cong \svect$. In particular, $\C$ is either non-degenerate or slightly degenerate.
\end{proposition}

\begin{proof} The category $\B = \D \cap \D'$ is the M\" uger center of $\D$, and it is a symmetric subcategory of $\C$. Since $\D$ is a characteristic subcategory, then so is $\B$. Therefore $\B$ must be weakly anisotropic. By Lemma \ref{wa-symm}, we obtain that $\B \cong \vect$ or $\B \cong \svect$. This implies the proposition.
\end{proof}

Since a braided fusion category of odd integer Frobenius-Perron dimension cannot be slightly degenerate, as a consequence of Proposition \ref{center-st}, we obtain:

\begin{corollary}\label{core-odd} Let $\C$ be a weakly anisotropic braided fusion category such that $\FPdim \C$ is an odd integer. Then $\C$ is non-degenerate. \qed
\end{corollary}

\begin{corollary}\label{core-odd-wgt} Let $\C$ be a braided fusion category of odd integer Frobenius-Perron dimension. Then the core of $\C$ is a non-degenerate braided fusion category.
\end{corollary}

\begin{proof} Since $\C$ has odd integer Frobenius-Perron dimension, then so does its core. The statement follows from Corollary \ref{core-odd}.
\end{proof}

\section{Tannakian subcategories of an integral weakly group-theoretical braided fusion category}\label{prev-core}

The main results of this section assert the  existence of nontrivial Tannakian subcategories in certain integral weakly group-theoretical braided fusion categories.
	
\begin{theorem}\label{wgt-tann} Let $\C$ be an integral  braided fusion category. Suppose that $\C$ is weakly group-theoretical. Then either $\C$ is pointed or it contains a nontrivial Tannakian subcategory.
\end{theorem}

\begin{proof} The proof is by induction on $\FPdim \C$. We may assume that $\C$ is not pointed and not Tannakian.
	
Suppose first that $\C_{ad} \subsetneq \C$. Since $\C_{ad}$ is weakly group-theoretical and integral, then by induction either it contains a nontrivial Tannakian subcategory, whence so does $\C$, or it is pointed. If $\C_{ad}$ is pointed, then $\C$ is nilpotent and since it is integral, then it is group-theoretical by \cite[Theorem 6.10]{dgno-nilpotent}. Since $\C$ is not pointed, it follows from \cite[Lemma 5.1]{witt-wgt} that $\C$ contains a nontrivial Tannakian subcategory.	

\medbreak We may then assume that $\C_{ad} = \C$, in other words, $\C$ admits no faithful group grading.

By \cite[Proposition 4.2]{ENO2}, there exist a series of fusion categories
\begin{equation}\label{cs-wgt}\vect = \C_0, \C_1, \dots , \C_n = \C,\end{equation}
and a series of finite groups $G_1, \dots, G_n$, such that, for all $1 \leq i \leq n$, the Drinfeld center $\Z(\C_i)$ contains a Tannakian subcategory $\E_i \cong \Rep G_i$ and there is an equivalence of braided fusion categories $(\E_i')_{G_i} \cong \Z(\C_{i-1})$, where $(\E_i')_{G_i}$ is the  de-equivariantization of the M\" uger centralizer $\E_i'$ in $\Z(\C_i)$ by $G_i$.

In particular, for all $i = 1, \dots, n$, $\Z(\C_i)_{G_i}$ is a braided $G_i$-crossed fusion category with neutral component $\Z(\C_i)_{G_i}^0 \cong \Z(\C_{i-1})$ and there is an equivalence of fusion categories $\Z(\C_i) \cong (\Z(\C_i)_{G_i})^{G_i}$. Let $F_i: \Z(\C_i) \to \Z(\C_i)_{G_i}$ denote the canonical tensor functor, so that $F_i$ is a normal tensor functor and $\mathfrak{Ker}_{F_i} = \E_i$.

\medbreak Suppose on the contrary that $\C$ contains no nontrivial Tannakian subcategory. Let us regard $\C$ as a fusion subcategory of $\Z(\C)$, by means of the canonical embedding $\C \to \Z(\C)$. Since $\C \cap \E_n$ is a Tannakian subcategory of $\C$, then $\C \cap \mathfrak{Ker}_{F_n} = \C \cap \E_n = \vect$. Therefore, by \cite[Lemma 2.3]{jh-wgt}, the restriction $F_n\vert_{\C}:\C \to \Z(\C_n)_{G_n}$ induces an equivalence of tensor categories $\C \cong F_n(\C) \subseteq \Z(\C_n)_{G_n}$.

Since by assumption $\C$ admits no faithful group grading, then $F_n(\C) \subseteq \Z(\C_n)_{G_n}^0 \cong \Z(\C_{n-1})$ (see Remark \ref{gding-sub}).
Hence, by \cite[Proposition 4.2]{jh-wgt}, $\C \subseteq (\E_n)'$ and $F_n$ induces by restriction an embedding of braided fusion categories $\C \to \Z(\C_n)_{G_n}^0 \cong \Z(\C_{n-1})$.

Iterating this argument, we obtain an embedding of braided fusion categories $\C \to \Z(\C_1)_{G_1}^0 \cong \Z(\C_0) = \vect$, which is a contradiction. The contradiction shows that $\C$ must contain a nontrivial Tannakian subcategory, as claimed. \end{proof}

\begin{corollary}\label{wgt-pt} Let $\C$ be an integral weakly group-theoretical braided fusion category and let $\Gamma$ be a subgroup of $\Aut_{br}\C$.
Suppose that $\C$ is $\Gamma$-anisotropic. Then either $\C$ is pointed or it contains a Tannakian subcategory of prime dimension. In particular, if $\C$ is not trivial, then
$\C_{pt} \neq \vect$.
\end{corollary}

\begin{proof} Assume $\C$ is not pointed. By Theorem \ref{wgt-tann}, $\C$ contains a nontrivial Tannakian subcategory, in particular $\Gamma \neq \{\id_\C \}$. Let $\E$ be a nontrivial Tannakian subcategory of $\C$ of minimal dimension. Then $\E$ does not contain any proper fusion subcategory, and therefore the same holds for all its $\Gamma$-conjugates. In particular $\E \cong \Rep G$, where $G$ is a finite simple group.
	
\medbreak
Since $\C$ is $\Gamma$-anisotropic there exists $\sigma \in \Gamma$ such that $\sigma(\E) \neq \E$, whence $\E \cap \sigma(\E) \cong \vect$, because $\E$ contains no proper fusion subcategories.

Consider the canonical normal tensor functor  $F: \C \to \C_G$, so that $\KER_F = \E$.  As $\E \cap \sigma(\E) \cong \vect$, it follows from \cite[Lemma 2.3]{jh-wgt} that $F$ induces by restriction an equivalence of fusion categories $\sigma(\E) \cong F\sigma(\E)$.

\medbreak Since $\sigma(\E)$ contains no proper fusion subcategories neither, then either $\sigma(\E) \cap \E' \cong \vect$ or $\sigma(\E) \subseteq \E'$.
Note that $\E' = F^{-1}(\C_G^0)$, so that the first possibility implies that $F\sigma(\E)$ is faithfully graded by a nontrivial subgroup of $G$, and therefore $F\sigma(\E)$ must be pointed of prime dimension. Hence the same holds for $\E \cong F\sigma(\E)$ and the lemma follows in this case.

\medbreak We may thus assume that $\sigma(\E) \subseteq \E'$. Then \cite[Proposition 7.7]{muegerII} implies that $\E \vee \sigma(\E) \cong \E \boxtimes \sigma(\E)$ as braided fusion categories. In particular, $\E \vee \sigma(\E) \cong \Rep (G \times G)$ is a Tannakian subcategory of $\C$.

\medbreak Since $\C$ is $\Gamma$-anisotropic, then the Tannakian subcategory $\E \vee \sigma(\E)$ cannot be $\Gamma$-stable and thus  there must exist $\sigma_2 \in \Gamma$ such that $\sigma_2(\E)$ is not contained in $\E \vee \sigma(\E)$.
Since $\sigma_2(\E) \cong \E$ contains no proper fusion subcategories, then  $\sigma_2(\E) \cap (\E \vee \sigma(\E)) \cong \vect$. As before, we have that $\sigma_2(\E) \cap (\E \vee \sigma(\E))' \cong \vect$ or $\sigma_2(\E) \subseteq (\E \vee \sigma(\E))'$.
The first possibility implies that $\sigma_2(\E)$ is faithfully graded by a nontrivial subgroup of $G\times G$, whence $\sigma_2(\E)$ is pointed of prime dimension, and the second possibility implies that $\E \vee \sigma(\E) \vee \sigma_2(\E) \cong \E \boxtimes \sigma(\E)\boxtimes \sigma_2(\E)$ as braided fusion categories and in particular $\E \vee \sigma(\E) \vee \sigma_2(\E) \cong \Rep (G \times G \times G)$ is a Tannakian subcategory of $\C$.

\medbreak
Assume that $\C$ contains no Tannakian subcategory of prime dimension. In view of the finiteness of the dimension of $\C$, continuing this process we find elements $\sigma = \sigma_1, \dots, \sigma_n$ of $\Gamma$ such that $$\E \vee \sigma_1(\E) \vee \dots \vee  \sigma_n(\E) \cong \E \boxtimes \sigma_1(\E)\boxtimes \dots \boxtimes \sigma_n(\E)$$ is a Tannakian subcategory of $\C$ and $\tau(\E) \subseteq \E \vee \sigma_1(\E) \vee \dots \vee  \sigma_n(\E)$, for all $\tau \in \Gamma$. This implies that $\E \vee \sigma_1(\E) \vee \dots \vee  \sigma_n(\E)$ is a $\Gamma$-stable (nontrivial) Tannakian subcategory, which contradicts the assumption on $\C$.

The contradiction comes from the assumption that $\C$ contains no Tannakian subcategory of prime dimension. Then we conclude that such a subcategory must exist and the lemma follows.	
\end{proof}

We point out that the conclusion of Corollary \ref{wgt-pt} may fail if $\C$ is not weakly anisotropic, as the following example shows.

\begin{example} Let $G$ be a finite nonabelian simple group $G$ and let $\C = \Z(\Rep G)$ be the Drinfeld center of the category $\Rep G$. The category $\C$ is integral weakly group-theoretical, and we have $\C_{pt} = \vect$.
	
Regard $\E = \Rep G$ as a Tannakian subcategory of $\C$ under the canonical embedding $\Rep G \to \Z(\Rep G)$. Then the de-equivariantization $\C_G$ is a pointed fusion category (equivalent to the category of finite-dimensional $G$-graded vector spaces).
	
In this example the category $\E$ is the unique nontrivial Tannakian subcategory of $\C$. In fact, if $\B \neq \E$ is another nontrivial Tannakian subcategory, then $\E \cap \B$ is a Tannakian subcategory of $\E$ implying that $\E \cap \B = \vect$ (because, $G$ being simple, the category $\Rep G$ contains no proper fusion subcategories). Let $F: \C \to \C_G$ denote the canonical normal tensor functor. It follows from \cite[Proposition 4.2]{jh-wgt}, that $F$ induces by restriction an equivalence of fusion categories between $\B$ and a  fusion subcategory of $\C_G$. Hence $\B$ must be pointed, which is a contradiction.
This shows that $\E$ is the unique Tannakian subcategory of $\C$, as claimed. 	
In particular, $\E$ is stable under all braided auto-equivalences of $\C$. This shows that $\C$ is not weakly anisotropic. 	
\end{example}

\section{The core of a weakly group-theoretical braided fusion category}\label{core}

The main goal of this section is to give a proof of Theorem \ref{core-wgt}. Our first two theorems regard the structure of weakly group-theoretical weakly anisotropic braided fusion categories.
	
\begin{theorem}\label{wa-wgt-int} Let $\C$ be a weakly group-theoretical integral braided fusion category such that $\C$ is weakly anisotropic. Then $\C$ is pointed.
\end{theorem}
	
\begin{proof} The proof is by induction on $\FPdim \C$. If $\FPdim \C = 1$, there is nothing to prove.  Assume that $\C$ is not trivial and the theorem holds for all braided fusion categories $\D$ such that $\FPdim \D < \FPdim \C$. Suppose first that $\C$ contains a proper non-degenerate characteristic fusion subcategory $\D \neq \vect$. Then  $\D'$ is also characteristic, and thus both $\D$ and $\D'$ must be weakly anisotropic (see Remark \ref{ppties} (iii)). By induction, $\D$ and $\D'$ are pointed. In addition, $\C \cong \D \boxtimes \D'$,  and hence $\C$ is pointed itself.

\medbreak 		
In view of Corollary \ref{wgt-pt}, $\C_{pt} \neq \vect$. Suppose that $\C_{pt}$ is non-degenerate. Since $\C_{pt}$ is characteristic, then we are done by the argument above.
Thus we may assume that $\C_{pt}$ is degenerate and then the M\" uger center $\E$ of $\C_{pt}$ is equivalent to $\svect$, by Proposition \ref{center-st}. Hence $\C_{pt}$ is slightly degenerate and by \cite[Proposition 2.6 (i)]{ENO2}, $\C_{pt} \cong \svect \boxtimes \C_0$ as braided fusion categories, where $\C_0$ is a pointed non-degenerate fusion subcategory.		
Then $\C \cong \C_0 \boxtimes (\C_0)'$ and we have $((\C_0)')_{pt} = (\C_0)' \cap \C_{pt} \cong \svect$.
We may assume that $(\C_0)'$ is not pointed, since otherwise $\C$ is pointed and we are done.

\medbreak
Let $\D = (\C_{pt})'$. Then $\D$ is characteristic and $\D \subseteq (\C_0)'$, because $\C_0$ is pointed. The assumption on $\C$ implies that $\D$ contains no nontrivial Tannakian subcategory stable under the subgroup $\Gamma = \Aut_{br}\C$ of $\Aut_{br}\D$. By Corollary \ref{wgt-pt} we obtain that either $\D$ is pointed or it  contains a Tannakian subcategory of prime dimension. The last possibility cannot hold because every fusion category of prime dimension is pointed, while $\D_{pt} \subseteq ((\C_0)')_{pt} \cong \svect$.
Therefore $\D$ must be pointed.

\medbreak It follows from \cite[Corollary 3.26]{DGNOI} that $\C_{ad} \subseteq (\C_{pt})' = \D$. Then $\C$ is nilpotent and therefore solvable, because it is braided \cite[Proposition 4.5 (iii)]{ENO2}.

\medbreak
Since the non-pointed fusion category $(\C_0)'$ is integral and weakly group-theoretical, it contains a nontrivial Tannakian subcategory $\E \cong \Rep G$, by Theorem \ref{wgt-tann}.
Since $\E$ is solvable, then $G$ is solvable and thus $\E_{pt} \neq \vect$. Hence $\E_{pt}$ is a nontrivial pointed Tannakian subcategory of $(\C_0)'$. This is impossible since $((\C_0)')_{pt} \cong \svect$. This contradiction shows that $\C$ must be pointed and finishes the proof of the theorem. \end{proof}

\begin{theorem}\label{wa-wgt} Let $\C$ be a weakly group-theoretical braided fusion category such that $\C$ is weakly anisotropic. Suppose $\C$ is not integral. Then there is an equivalence of braided fusion categories
$$\C \cong \I \boxtimes \D,$$
where $\I$ is an Ising braided category and $\D$ is a pointed weakly anisotropic braided fusion category.
\end{theorem}

\begin{proof} It will be enough to show that $\C \cong \I \boxtimes \D$, as braided fusion categories, where $\I$ is an Ising braided category and $\D$ is a pointed braided fusion category. In this case, $\I_{pt} = \I_{ad} = \C_{ad}$ is a  characteristic subcategory of $\C$. In addition $\D$ is the centralizer of $\I_{pt}$ in $\C$, because $\I$ is non-degenerate. Thus $\D$ is a characteristic subcategory, and  since $\C$ is weakly anisotropic, then so is $\D$.
	
\medbreak
The proof is by induction on $\FPdim \C$. Since $\C$ is not integral, then $\C_{ad} \neq \C$.
Furthermore,  $\C_{ad}$ is characteristic, and thus it must be weakly anisotropic (see Remark \ref{ppties} (iii)). In addition $\C_{ad}$ is integral \cite[Proposition 8.27]{ENO}; hence  it is pointed, by Proposition \ref{wa-wgt-int}. Moreover, $\C_{ad} \neq \vect$ because $\C$ is not pointed.

\medbreak
Let $\B$ denote the M\" uger center of $\C_{ad}$. By Proposition \ref{center-st},  $\B \cong \vect$ or $\B \cong \svect$.
In the first case $\C_{ad}$ is non-degenerate. Then $\C \cong \C_{ad} \boxtimes (\C_{ad})'$ as braided fusion categories, and $(\C_{ad})'$ is characteristic. By induction, $(\C_{ad})'$ is equivalent to a tensor product of an Ising braided category and a pointed braided fusion category. Then so is $\C$, since $\C_{ad}$ is pointed. We may thus assume that $\B \cong \svect$.

\medbreak
In view of \cite[Proposition 2.6 (ii)]{ENO2}, $\C_{ad} \cong \B \boxtimes \C_0$, where $\C_0$ is a  non-degenerate pointed braided category.
Then $\C \cong \C_0 \boxtimes (\C_0)'$ as braided fusion categories.

By \cite[Corollary 3.26]{DGNOI} $\C_{ad} \subseteq (\C_{pt})'$. Thus, since $\C_0 \subseteq \C_{pt}$, we get that $\C_{ad} \subseteq (\C_0)'$. But then $\C_0 \subseteq \C_{ad} \subseteq (\C_0)'$, which implies that $\C_0 \cong \vect$, because $\C_0$ is non-degenerate.  Therefore $\C_{ad} \cong \svect$.

\medbreak The maximal integral fusion subcategory $\C_{int}$ of $\C$ is weakly anisotropic (Remark \ref{ppties} (iii)). By Proposition \ref{wa-wgt-int}, $\C_{int}$ is pointed and therefore $\C_{int} = \C_{pt}$.

\medbreak Suppose that $X \in \C$ is a non-invertible simple object. Then $X \otimes X^* \in \C_{ad} \cong \svect$, and thus $\FPdim X = \sqrt 2$. Therefore the dimensional grading group of $\C$ is of order $2$. That is, for all non-invertible objects $X, Y$ of $\C$, we have $X \otimes Y \in \C_{int} = \C_{pt}$. Hence $\C$ has generalized Tambara-Yamagami fusion rules; see \cite[Section 5]{gen-ty}.

\medbreak
By Proposition \ref{center-st}, $\C' \cong \vect$ or $\C' \cong \svect$. In the first case, $\C$ is non-degenerate. Then, by \cite[Theorem 5.5]{gen-ty},  $\C \cong \I \boxtimes \D$ as a braided fusion category, where $\I$ is an Ising braided category and $\D$ is a non-degenerate pointed braided category. Hence we are done in this case.

\medbreak
It remains to consider the possibility $\C' \cong \svect$. In this case the  subcategory $\E = \C' \vee \C_{ad}$ is characteristic. If $\E$ is non-degenerate, then we are done by induction. So we may assume that $\E \cap \E' \cong \svect$, by Lemma \ref{wa-symm}. Since $\E$ is pointed, then $\E \cong \svect \boxtimes \E_0$, where $\E_0$ is a non-degenerate pointed braided fusion category.

Notice that, by \cite[Lemma 5.4]{mueger-galois}, $\C' \cap \C_{ad} \cong \vect$. Hence, by \cite[Corollary 3.12]{DGNOI}, $\FPdim \E = \FPdim \C' \, \FPdim \C_{ad} = 4$. Therefore $\FPdim \E_0 = 2$. Then $\E_0$ is the unique nontrivial non-degenerate subcategory of $\E$ (the remaining proper subcategories are $\C'$ and $\C_{ad}$, which are both equivalent to $\svect$) and therefore it must be characteristic. This implies that $(\E_0)'$ is also characteristic and in addition $\C \cong \E_0 \boxtimes (\E_0)'$. Hence the statement follows in this case by induction. This finishes the proof of the theorem.
\end{proof}

\begin{proof}[Proof of Theorem \ref{core-wgt}]
Let $\C$ be a weakly group-theoretical braided fusion category.
Then the core $\C_0$ of $\C$ is a weakly anisotropic braided fusion category and in addition it is also weakly group-theoretical.
It follows from Theorems \ref{wa-wgt-int} and \ref{wa-wgt} that $\C_0 \cong \B \boxtimes \D$, where $\B \cong \vect$ (if $\C_0$ is integral) or $\B \cong \I$, with $\I$ an Ising braided category (if $\C_0$ is not integral).

Moreover, if $\C$ is integral, then  $\C_0$ is integral as well; see \cite[Proposition 4.1]{witt-wgt}.  Then $\C_0$ is necessarily pointed, by Theorem \ref{wa-wgt-int}.
\end{proof}

\begin{corollary} Let $\C$ be weakly group-theoretical braided fusion category such that $\FPdim \C$ is odd. Then the core of $\C$ is a non-degenerate pointed weakly anisotropic braided fusion category.
\end{corollary}

\begin{proof} Since $\FPdim \C$ is odd, then $\C$ is integral. By Theorem \ref{core-wgt}, the core $\C_0$ of $\C$ is a pointed weakly anisotropic braided fusion category. By Corollary \ref{core-odd-wgt}, $\C_0$ is non-degenerate. This proves the corollary.
\end{proof}

\begin{corollary}\label{e'-gt} Let $\C$ be an integral weakly group-theoretical braided fusion category and
Let $\E \subseteq \C$ be a maximal Tannakian subcategory. Then $\E'$ is group-theoretical.	\end{corollary}

\begin{proof}
There is an equivalence of fusion categories $\E' \cong (\C_G^0)^G$. Since $\C_G^0$ is pointed, then $\E'$ is group-theoretical, as claimed.
\end{proof}

\begin{corollary} Let $\C$ be a weakly group-theoretical braided fusion category. Let $\E \cong \Rep G$ be a maximal Tannakian subcategory and let $\C_G$ be the corresponding $G$-crossed braided fusion category. Then the following hold:

\medbreak
(i) $\C_G$ is a 2-step nilpotent fusion category.

\medbreak
(ii) Suppose that $\C$ is integral and let $X$ be a simple object of $\C_G$. Then $|G| (\FPdim X)^2$ divides $\FPdim \C$. Moreover, if $\C$ is non-degenerate, then $|G|^2 (\FPdim X)^2$ divides $\FPdim \C$.
\end{corollary}

\begin{proof} The category $\C_G$ is an $H$-extension of the core $\C_G^0$, for some subgroup $H$ of $G$. Therefore $(\C_G)_{ad} \subseteq \C_G^0$ and, since $(\C_G)_{ad}$ is integral, then it is a pointed fusion category. This implies part (i).

Let $X$ be a simple object of $\C_G$. Since, by (i), $\C_G$ is nilpotent, then $(\FPdim X)^2$ divides $\FPdim (\C_G)_{ad}$ \cite[Corollary 5.3]{gel-nik}. Hence $(\FPdim X)^2$ divides $\FPdim \C_G^0$. Let $H$ be the (normal) subgroup of $G$ such that $\C_G$ is an $H$-extension of $\C_G^0$. Then  $\FPdim \C_G^0 = \FPdim \C / |G| |H|$. Notice that, if $\C$ is non-degenerate, then $H = G$. This implies part (ii) and finishes the proof of the corollary.
\end{proof}

\section{Solvability of a weakly group-theoretical braided fusion category}\label{s-solv-wgt}
As a consequence of Theorem \ref{core-wgt}, we obtain the following theorem that provides some criteria for the solvability of a weakly group-theoretical braided fusion category.

\begin{theorem}\label{cor-sol} Let $\C$ be a weakly group-theoretical braided fusion category. Then the following are equivalent:

\medbreak
(i) $\C$ is solvable

\medbreak
(ii) $\E$ is solvable, for some maximal Tannakian subcategory of $\C$.

\medbreak
(iii) $\E$ is solvable, for every Tannakian subcategory of $\C$.
\end{theorem}

\begin{proof} Since a fusion subcategory of a solvable fusion category is solvable, then (i) implies (ii) and (iii). It is clear that (iii) implies (ii). It will be enough to show that (ii) implies (i). Let $\E \cong \Rep G$ be a maximal Tannakian subcategory of $\C$ such that $\E$ is solvable. Then the group $G$ is solvable. It follows from Theorem \ref{core-wgt} that the core $\C_G^0$ of $\C$ is solvable. Then the de-equivariantization $\C_G$, being an $H$-extension of $\C_G^0$ for some subgroup $H$ of $G$, is solvable too. Since $\C \cong (\C _G)^G$, then $\C$ is solvable and we get (i).
\end{proof}

A natural number $d$ is said to \emph{force solvability} if any group of order $d$ is necessarily solvable. For instance, $d$ forces solvability if $d$ is not divisible by $4$ (by the Feit-Thompson Theorem), or if $d = p^aq^b$, $a, b \geq 0$ (by Burnside's Theorem).

\begin{corollary} Let $\C$ be a weakly group-theoretical non-degenerate braided fusion category. 	
Suppose that every natural number $d$ such that $d^2$ divides $\FPdim \C$ forces solvability. Then $\C$ is solvable. \end{corollary}

\begin{proof}
Recall that for every Tannakian subcategory $\E \subseteq \C$, we have that $(\FPdim \E)^2$ divides $\FPdim \C$.
Suppose that $\E \subseteq \C$ is any Tannakian subcategory and let $G$ be a finite group such that $\E \cong \Rep G$. Then $|G|^2 = (\FPdim \E)^2$ divides $\FPdim \C$ and by assumption $G$ is solvable, whence so is $\E$. It follows from Theorem  \ref{cor-sol} that $\C$ is solvable, as claimed.
\end{proof}

\begin{proposition}\label{paqbd} Let $p$ and $q$ be prime numbers. Let $\C$ be a non-degenerate braided fusion category such that $\FPdim \C$ is an integer and suppose that for every simple object $X$ of $\C_{int}$, there exist non-negative integers $a, b$ such that $\FPdim X = p^aq^b$. Then $\C$ is solvable.
\end{proposition}

\begin{proof} Suppose that $G$ is a finite group such that the degree of every irreducible representation of $G$ is of the form $p^aq^b$, for some non-negative integers $a, b \geq 0$. It follows from the Ito-Michler's Theorem  \cite[Theorem 5.4]{michler} that if $r \neq p, q$ is a prime divisor of the order of $G$, then the $r$-Sylow subgroup of $G$ is normal and abelian. Hence $G$ isomorphic to a semidirect product $G_1 \rtimes G_2$, where $G_1$ is an abelian group of order relatively prime to $pq$ and $G_2$ is a group of order $p^nq^m$, $n, m \geq 0$. By Burnside's Theorem, $G_2$ is solvable and therefore so is $G$. In view of the assumptions on $\C$, this shows that every Tannakian subcategory $\E \subseteq \C$ is solvable. Thus, by Theorem \ref{cor-sol}, it will be enough to show that $\C$ is weakly group-theoretical.
	
\medbreak
The proof is by induction on $\FPdim \C$. We may assume that $\C$ is not nilpotent (and in particular not pointed). We may also assume that $\C$ contains no proper non-degenerate fusion subcategory $\D$; otherwise, $\C \cong \D \boxtimes \D'$ and by induction, $\D$ and $\D'$ are weakly group-theoretical, whence so is $\C$.
	
Observe that if $\E \cong \Rep G$  is a Tannakian subcategory of $\C$, then the de-equivariantization $\C_G$ has integer Frobenius-Perron dimension and for every simple object $Y$ of $(\C_G)_{int}$ we also have  $\FPdim Y = p^rq^s$, for some non-negative integers $r, s \geq 0$ (see Subsection \ref{ss-equiv}). If $\E$ is not trivial, then $\FPdim \C_G^0 < \FPdim \C$ and since $\C_G^0$ is non-degenerate, then it is weakly group-theoretical by induction, and hence so is $\C$.
Therefore it will be enough to show that $\C$ contains a nontrivial Tannakian subcategory.
	
\medbreak Suppose first that $\C_{ad} \neq \C$. We may assume that $\C_{ad}$ contains no nontrivial non-degenerate or Tannakian subcategory. By \cite[Lemma 7.1]{witt-wgt}, we conclude that $\C_{ad}$ is
slightly degenerate and $G[X] = \1$, for all simple object  $X \in \C_{ad}$.
	
In addition, $\C_{ad}$ is integral. Therefore, for every simple object $X$ of $\C_{ad}$, we have $\FPdim X = p^nq^m$, for some $n, m \geq 0$. Moreover, we may assume that $\C_{ad}$ is not pointed, since otherwise $\C$ is nilpotent and we are done.
	
If $\C_{ad}$ has a simple object of odd prime power dimension then, since it is integral and slightly degenerate, it contains a
nontrivial Tannakian subcategory by \cite[Proposition 7.4]{ENO2}. Hence we may assume that this is not the case.
	
If $\FPdim X$ is divisible by $pq$ for all non-invertible simple object $X \in \C_{ad}$, then $pq$ divides the order of the group $G[X]$ for all such simple objects, which is a contradiction. Then this possibility is discarded.
	
It remains to consider the case where $\FPdim X = 2^m$, $m \geq 0$, for every
simple object $X$ of $\C_{ad}$. In this case, \cite[Theorem 7.2]{witt-wgt} implies that $\C_{ad}$ is solvable and therefore $\C$, being a group extension of $\C_{ad}$ is weakly group-theoretical.
	
\medbreak Suppose next that $\C = \C_{ad}$. In particular $\C$ is integral and the Frobenius-Perron dimensions of simple objects of $\C$ are of the form $p^nq^m$, $n, m \geq 0$.
Since $\C$ is non-degenerate, then $\C_{pt} = (\C_{ad})'$ is the trivial fusion subcategory. It follows that $\C$ must contain a simple object of positive prime power dimension. Since $\C$ is integral, this implies that $\C$ contains a nontrivial symmetric subcategory $\E$, by  \cite[Corollary 7.2]{ENO2}. Since $\C$ does not contain nontrivial Tannakian subcategories, we must have $\E \cong \svect$. But then $\E \subseteq \C_{pt}$, which is a contradiction. The contradiction shows that in this case $\C$ must contain a nontrivial Tannakian subcategory, as claimed.
This finishes the proof of the proposition.
\end{proof}

Suppose that $\C$ is a non-degenerate braided fusion category of Frobenius-Perron dimension $p^aq^bd$, where $p$ and $q$ are prime numbers and $d$ is a square-free natural number. Since $\C$ is non-degenerate then, for every simple object $X$ of $\C$, the $(\FPdim X)^2$ divides $\FPdim \C$ \cite[Theorem 2.11]{ENO2}. Therefore for every simple object $X$ of $\C_{int}$, $\FPdim X = p^nq^m$, for some $n, m \geq 0$. From Proposition \ref{paqbd}, we obtain the following corollary, that strengthness the statement in \cite[Theorem 7.4]{witt-wgt}.

\begin{corollary}\label{paqbd-solv} Let $p$ and $q$ be prime numbers and let $d$ be a square-free natural number. Let $\C$ be a non-degenerate braided fusion category such that $\FPdim \C = p^aq^bd$, $a, b \geq 0$. Then $\C$ is solvable. \qed
\end{corollary}

\begin{corollary} Let $\C$ be an integral non-degenerate braided fusion category of dimension $p^aq^bd$, $a, b \geq 0$. Then the core of $\C$ is a pointed weakly anisotropic braided fusion category. \qed
\end{corollary}

Let $\C$ and $\D$ be spherical fusion categories. Recall from \cite[Section 6]{fr-solvable}, that $\C$ and $\D$ are \emph{$S$-equivalent} if there exists a bijection $f: \Irr(\Z(\C)) \to \Irr(\Z(\D))$, called an $S$-equivalence, such that $f(\1) = \1$ and $S_{f(X), f(Y)} = S_{X, Y}$, for all $X, Y \in \Irr(\Z(\C))$. Here, $\Z(\C)$ denotes the Drinfeld center of $\C$.

The following theorem extends a result in \cite[Theorem 6.5]{fr-solvable}.

\begin{theorem}\label{fr-wgt} Let $\C$ be a weakly group-theoretical braided fusion category. Suppose that $\C$ is $S$-equivalent to a solvable fusion category $\D$. Then $\C$ is solvable.
\end{theorem}

Recall that every weakly group-theoretical fusion category $\C$ has integer Frobenius-Perron dimension and it is therefore spherical. Moreover, there is a canonical positive spherical structure on $\C$ (that is, the unique spherical structure with respect to which quantum dimensions coincide with Frobenius-Perron dimensions), see \cite[Propositions 8.23 and 8.24]{ENO}.

\begin{proof} Note that the Drinfeld centers $\Z(\C)$ and $\Z(\D)$ are both weakly group-theoretical non-degenerate braided fusion categories.
	
Let $f: \Irr(\Z(\C)) \to \Irr(\Z(\D))$ be an $S$-equivalence. Then $f$ is a Grothen\-dieck equivalence, that is, it preserves fusion rules. In addition, $f$ preserves Frobenius-Perron dimensions and centralizers; see \cite[Lemma 6.2]{fr-solvable}.  Hence, $f$ induces an inclusion preserving bijection, that we shall still denote by $f$, between the lattices of fusion subcategories of $\Z(\C)$ and $\Z(\D)$. Suppose that $\D$ is solvable, thus $\Z(\D)$ is solvable as well.
	
Let $\E \subseteq \Z(\C)$ be a Tannakian subcategory and let $G$ be a finite group such that $\E \cong \Rep G$ as braided fusion categories. Then $f(\E) \subseteq \Z(\D)$ is a symmetric subcategory, and therefore there is an equivalence of fusion categories  $f(\E) \cong \Rep L$, for some finite group $L$. Since $\D$ is solvable, then $f(\E)$ is solvable and thus the group $L$ is solvable. This implies that the group $G$ is solvable, because the categories $\E \cong \Rep G$ and $f(\E) \cong \Rep L$ have the same fusion rules (hence $G$ and $L$ have the same character table). Therefore $\E$ is solvable. In view of Theorem \ref{cor-sol}, this implies that $\Z(\C)$ is solvable, and therefore so is $\C$. This finishes the proof of the theorem.
\end{proof}

\section{Integral modular categories with simple objects of Frobenius-Perron dimension at most $2$}\label{s-cd12}	

In this section we give a proof of Theorem \ref{cd12}. Along this section $\C$ will be an integral non-degenerate braided fusion category such that $\FPdim X \leq 2$, for every simple object $X$ of $\C$, in other words, we have  $\cd (\C) \subseteq  \{ 1, 2 \}$.

\medbreak
In view of \cite{cd2}, the assumption implies that $\C$ is solvable. By Theorem \ref{core-wgt}, we obtain:

\begin{corollary} The core of $\C$ is a pointed non-degenerate braided fusion category. \qed
\end{corollary}

\medbreak
\emph{We shall assume in what follows that $\C$ is not group-theoretical. In particular, $\C$ is not pointed, that is, $\cd (\C) =  \{ 1, 2 \}$.}

\medbreak
Let $\E \subseteq \C$ be a maximal Tannakian subcategory. Let also $G$ be a finite group such that $\Rep G \cong \E$ and let $\D = \C_G$ denote the de-equivariantization. Recall from Subsection \ref{ss-core} that the de-equivariantization $(\E')_G$ coincides with the core of $\C$ and $\D$ is a $G$-extension of $\C_G^0 \cong (\E')_G$. In addition, we have an equivalence of braided fusion categories
\begin{equation}\label{zdeu}\C \boxtimes (\C_G^0)^{rev} \cong \Z(\D).\end{equation}

Therefore we also have
\begin{equation}\cd (\Z(\D)) = \{1, 2\}.\end{equation}

\begin{lemma}\label{gama} The adjoint subcategory $\D_{ad}$ is pointed and $G(\D_{ad})$ is an abelian $2$-group.
\end{lemma}

\begin{proof} Since $\D$ is a graded extension of the pointed subcategory $\C_G^0$, then the adjoint subcategory $\D_{ad}$ is contained in $\C_G^0$. Therefore $\D_{ad}$ must be pointed with an abelian group of isomorphism classes of invertible objects, because $\C_G^0$ is a pointed braided fusion category.
	
Notice that, since $\D = \C_G$, then $\cd(\D) = \{1, 2\}$. On the other hand, since $\D_{ad}$ is pointed, then for every 2-dimensional simple object $S$ of $\D$ we have a decomposition $S \otimes S^* \cong \bigoplus_{g \in G[S]}g$, where $G[S] \subseteq G(\D)$ is the subgroup consisting of all those $g$ such that $g \otimes S \cong S$.
In particular $|G[S]| = 4$. In addition, the group  $G(\D_{ad})$ is generated by the subgroups $G[S]$, then the lemma follows.
\end{proof}

Let $U = U(\D)$ be the universal grading group of the category $\D$. Then $\Rep U$ is a Tannakian subcategory of $\Z(\D)$ and $\Z(\D)_U$ is a $U$-extension of the Drinfeld center $\Z(\D_{ad})$. See \cite{center-gdd}.

It follows from Lemma \ref{gama} that $\Z(\D_{ad})$ is a (non-degenerate) braided 2-category.

\begin{lemma}\label{uabel} The group $U$ is abelian.
\end{lemma}

\begin{proof} Observe that the de-equivariantization $\Z(\D)_U$ also satisfies the condition $\cd (\C_G) = \{1, 2\}$. In fact, since $\Z(\D) \cong (\Z(\D)_U)^U$ then, it follows from the description of simple objects in an equivariantization in \cite[Corollary 2.13]{fr-equiv}, that $\Z(\D)_U$ is integral and  $\FPdim Y \leq 2$, for every simple object of $\Z(\D)_U$. Since, by assumption, $\C$ is not group-theoretical, then \eqref{zdeu} implies that $\Z(\D)$ is not group-theoretical neither, and thus $\Z(\D)_U$ is not pointed.
	
\medbreak
Furthermore, isomorphism classes of simple objects of $\Z(\D)$ are 	parameterized by
pairs $(Y, \pi)$, where $Y$ runs over the orbits of the action of $U$ on $\Irr(\Z(\D)_U)$ and $\pi$ runs over the equivalence classes of irreducible $\alpha_Y$-projective representations of the inertia subgroup $U_Y \subseteq U$, for certain 2-cocycle $\alpha_Y: U_Y \times U_Y \to k^{\times}$.
	
In addition, if $S_{Y, \pi}$ is a simple object corresponding to such pair $(Y, \pi)$, then $$\FPdim S_{Y, \pi} = \dim \pi \, [U: U_Y] \, \FPdim Y.$$
	
\medbreak
Hence  $\FPdim S_{Y, \pi} = 1$ or $2$, for all such pair $(Y, \pi)$.
Thus, if $\FPdim Y = 2$, then $U = U_Y$ and $\dim \pi = 1$, for every 	irreducible $\alpha_Y$-projective representations of $U = U_Y$. This implies that the cohomology class of $\alpha_Y$ is trivial and moreover, $U = U_Y$ is abelian, as claimed. \end{proof}
	
As a consequence of the previous lemma, we obtain:

\begin{corollary}\label{zd-ext} The category $\Z(\D)$ is a $U$-extension of $(\Rep U)'$.
\end{corollary}

\begin{proof}Recall that $\Z(\D)_U = \bigoplus_{u \in U}(\Z(\D)_U)_u$ is a $U$-graded extension of the category $\Z(\D)_U^0$.
Since by Lemma \ref{uabel} the group $U$ is abelian, then the adjoint action of $U$ on itself is trivial. Hence there is an induced $U$-grading on $\Z(\D)$, $\Z(\D) = \bigoplus_{u \in U}\Z(\D)_u$, such that $\Z(\D)_e = (\Z(\D)_U^0)^U = (\Rep U)'$; see \cite[Proposition 3.28]{mueger-crossed}.
\end{proof}

\begin{lemma}\label{repu-nil} The category $(\Rep U)'$ is nilpotent.
\end{lemma}

\begin{proof} There are equivalences of braided fusion categories $$(\Rep U)' \cong (\Z(\D)_U^0)^U \cong \Z(\D_{ad})^U.$$ By Lemma \ref{uabel}, $U$ is abelian. On the other hand, by Lemma \ref{gama}, $\D_{ad}$ is a $2$-category and therefore so is its Drinfeld center $\Z(\D_{ad})$. The lemma follows from Lemma \ref{equiv-nil}.
\end{proof}

\begin{proof}[Proof of Theorem \ref{cd12}] Let $\C$ be an integral non-degenerate braided fusion category such that $\FPdim X \leq 2$, for every simple object $X$ of $\C$.
	
Let $\Rep G \cong \E \subseteq \C$ be a maximal Tannakian subcategory and let $\D = \C_G$ the corresponding de-equivariantization. Corollary \ref{zd-ext} and Lemma \ref{repu-nil} imply that, unless $\C$ is group-theoretical, $\Z(\D)$ is nilpotent. Therefore so is $\C$, in view of the equivalence \eqref{zdeu}. Then $\C$ must be group-theoretical, by \cite[Corollary 9.4]{dgno-nilpotent}.  This finishes the proof of the theorem.	
\end{proof}
	
\bibliographystyle{amsalpha}

\end{document}